\documentclass{amsart}

\usepackage{graphicx}
\usepackage{amsmath,latexsym}
\usepackage{amsfonts,amssymb}
\usepackage{amscd, amsthm}
\usepackage{stmaryrd}
\usepackage{fancyhdr}
\usepackage{commath,enumitem,chngcntr}
\usepackage{indentfirst, hyperref}
\usepackage[titletoc,toc,title]{appendix}
\usepackage{bbm}
\usepackage{pst-node}
\usepackage{tikz-cd} 
\usepackage{pgf,tikz}
\usepackage{mathrsfs}
\usetikzlibrary{arrows,calc,positioning}
\usepackage{float}
\usepackage{subcaption}

\newtheorem{theorem}{Theorem}[section]
\newtheorem{lemma}[theorem]{Lemma}
\newtheorem{definition}[theorem]{Definition}

\newtheorem{proposition}[theorem]{Proposition}
\newtheorem{corollary}[theorem]{Corollary}

\begin{document}
\title{Spaces of Geodesic Triangulations of Surfaces}

\author{Yanwen Luo}
\address{Department of Mathematics, University of California, Davis
, California 95616}
\email{ywluo@ucdavis.edu}
\thanks{The author was supported in part by NSF Grant DMS-1719582.}


\begin{abstract}

We give a short proof of the contractibility of the space of geodesic triangulations with fixed combinatorial type of a convex polygon in the Euclidean plane. Moreover, for any $n>0$, we show that there exists a space of geodesic triangulations of a polygon with a triangulation, whose $n$-th homotopy group is not trivial. 
\end{abstract}

\maketitle

\section{Introduction}
This paper provides two results concerning the space of geodesic triangulations of a planar polygon. We first give a short new proof of the contractibility of the space of geodesic triangulations with fixed combinatorial type of a convex polygon, originally proved by Bloch, Connelly, and Henderson \cite{bloch1984space}, following a series of partial results \cite{bing1978linear, cairns1944deformations, cairns1944isotopic, ho1973certain}. We then consider the homotopy groups of the space of geodesic triangulations of a non-convex polygon. We show that each homotopy group can be non-trivial. This answers an open question asked in 1980 \cite{connelly1983problems}.

\begin{figure}[H]
  \includegraphics[width=0.7\linewidth]{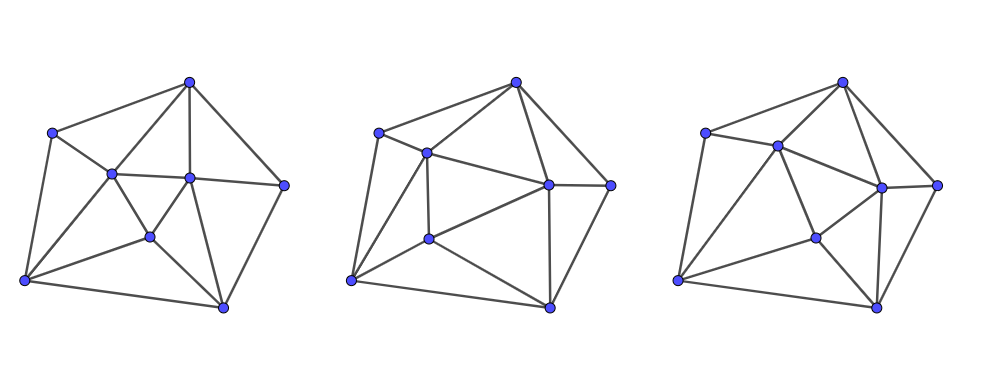}
  \caption{Geodesic triangulations of a polygon. }
\end{figure}

An embedded $n$-sided polygon $\Omega$ in the plane is determined by a map $\phi$ from the set $V_B = \{v_1,v_2, \cdots, v_n, v_{n+1} = v_1\}$ to $\mathbb{R}^2$, and line segments connecting the images under $\phi$ of two consecutive vertices in $V_B$ . We assume in the rest of this paper that $T = (V, E, F)$ is a triangulation of $\Omega$ with vertices $V$, edges $E$ and faces $F$ such that   $V = V_I\cup V_B$, where $V_I$ is the set of interior vertices and $V_B$ is the set of boundary vertices.

A  \textit{geodesic triangulation} with combinatorial type  $T$ of $\Omega$ is an embedding $\psi$ of the $1$-skeleton of $T$ in the plane  such that $\psi$ agrees with $\phi$ on $V_B$, and $\psi$ maps every edge in $E$ to a line segment parametrized by arc length. The set of these maps is called the \textit{space of geodesic triangulations on $\Omega$ with combinatorial type $T$}, and is denoted by $X(\Omega, T)$. Each geodesic triangulation is uniquely determined by the positions of the interior vertices in $V_I$, so the topology of $X(\Omega, T)$ is induced by $\Omega^{|V_I|} \subset \mathbb{R}^{2|V_I|}$. The main result of this paper is

\begin{theorem}
For any integer $n>0$, there exists a planar polygon $\Omega$ with a triangulation $T$ of $\Omega$ such that the space $X(\Omega, T)$ of geodesic triangulations with combinatorial type $T$ of $\Omega$ has non-trivial $n$-th homotopy group.
\end{theorem}

Notice that $X(\Omega, T)$ could be empty if the boundary is complicated. For instance, if the polygon is not star-shaped, then there is no geodesic  triangulation with only one interior vertex. 

Ho \cite{ho1973certain} showed that the space $X(\Omega, T)$ is equivalent to the space $L(\Omega, T)$ of simplexwise linear homeomorphisms of $\Omega$ with triangulation $T$. The homotopy type of the space $L(\Omega, T)$ has been studied in \cite{bing1978linear, bloch1984space, cairns1944isotopic, ho1973certain}, because it is closely related to the problem of existence and uniqueness of differentiable structures on triangulated manifolds. In addition, the path-connectedness of $L(\Omega, T)$ has implications to graph morphing problems in computational geometry \cite{de2003tutte, floater1999morph, surazhsky2001controllable, surazhsky2003intrinsic}.

Cairns \cite{cairns1944deformations, cairns1944isotopic} initiated an investigation of  the topology of the space of geodesic triangulations of a geometric triangle in the Euclidean plane and the round 2-sphere.
\begin{theorem}[Cairns \cite{cairns1944isotopic}]
If $\Omega$ is a geometric triangle with a triangulation $T$ in the plane, then $X(\Omega, T)$ is path-connected.
\end{theorem}
Ho \cite{ho1973certain} then proved that this space is simply-connected.

\begin{theorem}[Ho \cite{ho1973certain}]
If $\Omega$ is a geometric triangle with a triangulation $T$ in the plane, then $X(\Omega, T)$ is simply-connected.
\end{theorem}

Bloch, Connelly, and Henderson \cite{bloch1984space} extended these results to general convex polygons and further proved the contractibility of the space of simplexwise linear homeomorphisms of a convex polygon. In a recent paper, Cerf \cite{1910.00240} improved the original argument in \cite{bloch1984space}.

\begin{theorem}[Bloch, Connelly, and Henderson \cite{bloch1984space}]
If $\Omega$ is a convex polygon with a triangulation $T$ in the plane, then $X(\Omega, T)$ is homeomorphic to $\mathbb{R}^{2|V_I|}$.
\end{theorem}

Theorem 1.5 can be regarded as a discrete version of the classical theorem due to Smale \cite{smale1959diffeomorphisms} stating that the group of diffeomorphisms of the 2-disk fixing the boundary pointwise is contractible. As pointed out in \cite{bloch1984space}, Theorem 1.5 leads to an alternative proof of Smale's theorem.

Bing and Starbird \cite{bing1978linear} considered the more general case of star-shaped polygons. A \textit{dividing edge} in a triangulation $T$ is an interior edge connecting two boundary vertices. 

\begin{theorem}[Bing and Starbird \cite{bing1978linear}]
If $\Omega$ is a star-shaped polygon with a triangulation $T$ in the plane, and $T$ does not contain any dividing edge,  then $X(\Omega, T)$ is non-empty and path-connected.
\end{theorem}

Bing and Starbird \cite{bing1978linear} also showed that $X(\Omega, T)$ is not necessarily path-connected if the boundary is not star-shaped. 

All the results above were proved using induction. In Theorem 3.5, we will provide a constructive proof  based on Tutte's embedding theorem. On the other hand, Theorem 1.1 shows that this result does not extend to $X(\Omega, T)$ for non-convex polygons.

\noindent\textbf{Organization of this paper. } In Section 2, we recall Tutte's embedding theorem and its generalizations. In Section 3, we give a new proof of the  contractibility of $X(\Omega, T)$ when $\Omega$ is convex, based on Tutte's method. In Section 4, we prove Theorem 1.1. In Section 5, we discuss some conjectures about the homotopy types of spaces of geodesic triangulations of general surfaces.

\noindent\textbf{Acknowledgement.} This work is in partially supported by the NSF Grant DMS-1719582. The author would like to thank his advisor, Professor Joel Hass, for suggesting this problem, insightful discussions, and constant encouragement.

\section{Tutte's embedding and its generalization}
\subsection{Tutte's embedding for the disk}

Given a triangulation $T = (V, E, F)$ of the 2-disk with vertices $V$, edges $E$ and faces $F$,  the $1$-skeleton of $T$ is a planar graph. Tutte \cite{tutte1963draw} provided a constructive method to generate a straight-line embedding of a 3-vertex-connected planar graph shown in Figure 2. The procedure starts by setting one face of the graph as a convex polygon, then solves for the coordinates of the other vertices with a system of linear equations. 

Using a discrete maximum principle, Floater \cite{floater2003one} extended Tutte's result for the case of triangulations of the 2-disk.

\begin{figure}[h!]
  \includegraphics[width=0.5\linewidth]{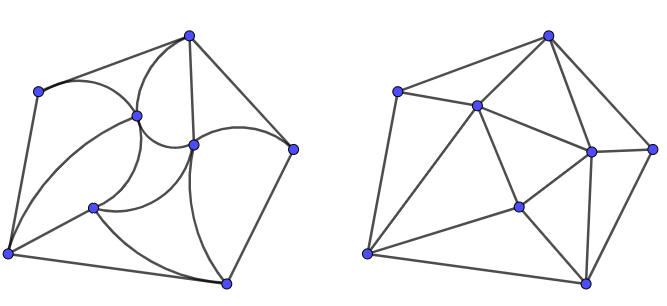}
  \caption{Tutte's embedding}
\end{figure}
\begin{theorem}[Floater \cite{floater2003one}]
	Assume $T = (V, E, F)$ is a triangulation of a convex $n$-sided polygon $\Omega$, and $\psi$ is a simplexwise linear homeomorphism from $T$ to $\mathbb{R}^2$. If  $\psi$ maps every interior vertex in $T$ into the convex hull of the images of its neighbors, and maps the cyclically ordered boundary vertices $v_1, v_2, \cdots, v_n$ of $T$ to the cyclically ordered vertices of $\Omega$, then $\psi$ is one to one.
\end{theorem}
Theorem 2.1 is a discrete version of the Rado-Kneser-Choquet theorem \cite{duren2004harmonic}, which states that a smooth harmonic map from the unit 2-disk to a convex domain bounded by a Jordan curve in the plane is homeomorphic, when its restriction to the boundary of the 2-disk is homeomorphic. Moreover, it gives a constructive method to produce geodesic triangulations of a convex polygon with the combinatorial type of $T$ as follows. 

\textbf{Step 1.} Assign a positive weight $c_{ij}$ to a directed edge $(i, j) \in \vec{E}$, where $\vec{E}$ is the set of directed edges of $T$. Then normalize the weights to make the sum of all outgoing weights around each interior vertex equal to $1$
$$w_{ij} = \frac{c_{ij}}{\sum_{j\in N(v_i)} c_{ij}}, \quad v_i\in V_I.$$
The set $N(v_i)$ consists of all the vertices that are neighbors of the vertex $v_i\in V_I$. Notice that we don't impose symmetry condition $w_{ij} = w_{ji}$.

\textbf{Step 2.} Fix the coordinates of boundary vertices $V_B$, which together form a  convex $n$-sided polygon $\Omega$,

$$\begin{pmatrix}
x_i \\
y_i
\end{pmatrix}   = \phi(v_i) = 
\begin{pmatrix}
b^x_i \\
b^y_i
\end{pmatrix}, \quad v_i\in V_B.
$$

The coordinates of vertices in $V_B$ are determined by the map $\phi$.

\textbf{Step 3.} Solve the coordinates for interior vertices with boundary coordinates given by Step 2,

    $$\sum_{j\in N(v_i)}w_{ij}\begin{pmatrix}
x_j \\
y_j
\end{pmatrix}
=
\begin{pmatrix}
x_i \\
y_i
\end{pmatrix}, \quad v_i\in V_I.$$

\textbf{Step 4.} Put the vertices in the positions given by these coordinates, and connect the vertices with line segments based on the combinatorics of the triangulation $T$.

Theorem 2.1 states that the result is a geodesic triangulation of $\Omega$ with the combinatorial type of $T$. The linear system in Step 3 implies that the $x$-coordinate (or $y$-coordinate) of one interior vertex is a convex combination of the $x$-coordinates (or $y$-coordinates) of its neighbors.  The coefficient matrix  of this system is not necessarily symmetric but it is diagonally dominant, so the solution exists and is unique. This procedure is called \textit{Tutte's method}.

This method has been generalized to surfaces with non-trivial topologies. Colin de Verdiere \cite{de1991comment} and Hass and Scott \cite{hass2012simplicial} showed that every triangulation of a closed surface with a metric of non-positive curvature can be realized as a geodesic triangulation. Gortler, Gotsman, and Thurston \cite{gortler2006discrete} reproved Tutte's theorem using \textit{discrete one forms} and generalized it to the cases of flat tori and multiple-connected polygonal regions, with appropriate assumptions on the boundaries. Aigerman and Lipman \cite{aigerman2015orbifold} further extended this method to Euclidean orbifolds with spherical topology.

\section{Geodesic Triangulations of the 2-Disk with Convex Boundary}

In this section, we present a short proof of the contractibility of the space of geodesic triangulations on a convex polygon, based on Tutte's method. Let us consider the topology of the space $X(\Omega, T)$ where $\Omega$ is a fixed convex polygon in the plane. Let $E_I$ be the set of interior edges in $T$ and $E_B$ be the set of boundary edges in $T$. 

\begin{definition}
Assume $\Omega$ is a convex polygon with a triangulation $T$. A collection of weights, defined by a $|V_I|\times |V|$ matrix $(w_{ij})$, is \textit{permissible} if 
\begin{itemize}

\item $w_{ii} = 1$ for all $i = 1, 2, \cdots, |V_I|$;
	\item $w_{ij} = 0$ if $v_i$ is not connected to $v_j$;
	\item $w_{ij} > 0$ if $v_i$ is connected to $v_j$;
	\item $\sum_{j\in N(v_i)}w_{ij} = 1$ for each interior vertex $v_i$. 

\end{itemize}
Define $W(\Omega, T)$ to be  the space of permissible weights of $(\Omega, T)$. 
\end{definition}


\begin{definition}
The \textit{Tutte map} $\Psi: W(\Omega, T) \to X(\Omega, T)$ sends a collection of permissible weights $(w_{ij})$ in $W(\Omega, T)$ to the unique geodesic triangulation $\tau \in X(\Omega, T)$ determined by the solution to the linear system in Step 2 and Step 3 of Tutte's method, with coefficients $(w_{ij})$ and boundary vertices of $\Omega$ determined by $\phi$.
\end{definition}

The space $W(\Omega, T)$ is a $2|E_I| -|V_I|$ dimensional manifold.  The range $X(\Omega, T)$ is a $2|V_I|$ dimensional manifold. One can deduce that $|E_I| - 3|V_I| = |E_B| - 3$. Hence the dimension of $W(\Omega, T)$ is not less than the dimension of $X(\Omega, T)$, with equality when the boundary is a triangle.  
\begin{lemma}
	The Tutte map $\Psi$ is continuous and surjective from $W(\Omega, T)$ to $X(\Omega, T)$.
\end{lemma}
\begin{proof}
	By Theorem 2.1, for any $(w_{ij})\in W(\Omega, T)$, the solution to the linear system generates a geodesic triangulation of $T$, so $\Psi$ is well-defined. The continuity follows from the continuous dependence on the coefficients of the solutions to the linear system. To show surjectivity, given a geodesic triangulation $\tau\in X(\Omega, T)$, any interior vertex $v_i$ in $\tau$ is in the convex hull of its neighbors. Then we can construct weights $(w_{ij})$ for a geodesic triangulation $\tau$ using the \textit{mean value coordinates} defined in \cite{floater2003mean} below.
	
\begin{figure}[h!]
  \includegraphics[width=0.3\linewidth]{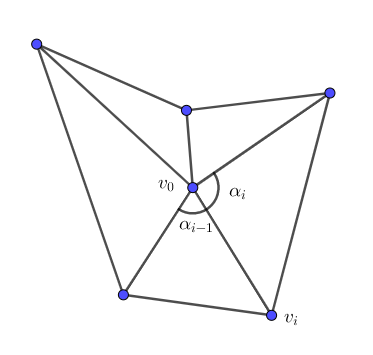}
  \caption{The mean value coordinate at $v_0$}
\end{figure}

	The mean value coordinates on the directed edges of a geodesic triangulation are given by 
$$w_{ij} = \frac{c_{ij}}{\sum_{j\in N(v_i)} c_{ij}}, \text{where} \quad c_{ij} = \frac{\tan (\alpha^j_{i-1}/2) +\tan (\alpha^j_i/2) }{||v_i - v_j||}, \quad (i, j)\in \vec{E}.$$ 
The two angles $\alpha^j_{i-1}$ and $\alpha^j_i$ at $v_i$ share the edge $(i, j)\in \vec{E}$ in the Figure 3. The mean value coordinates provide a smooth map from $X(\Omega, T)$ to $W(\Omega, T)$.
\end{proof}

Floater \cite{floater1999morph} proposed another construction of weights by taking the average of barycentric coordinates. An alternative method is to take the center of mass of the space of weights $(w_{ij})\in W$ such that $\Psi((w_{ij})) = \tau$. All three methods agree with the barycentric coordinates of a vertex when the star of this vertex is a triangle. 

\begin{definition}
The map $\sigma: X(\Omega, T) \to W(\Omega, T)$ sends a geodesic triangulation $\tau$ to a permissible weight $(w_{ij})$ in $W(\Omega, T)$ determined by the mean value coordinates.
\end{definition}

\begin{theorem}
If $\Omega$ is a convex polygon in $\mathbb{R}^2$ with a triangulation $T$, the space of geodesic triangulations $X(\Omega, T)$ is contractible. 
\end{theorem}
\begin{proof}
	The map $\sigma$ is continuous, and $\Psi(\sigma(\tau)) = \tau$ for any $\tau \in X(\Omega, T)$, so the map $\sigma$ is a global section of $\Psi$ from $X(\Omega, T)$ to $W(\Omega, T)$. Since $W(\Omega, T)$ is convex, there exists a linear isotopy $(1 - t)\sigma\circ\Psi + t\mathbf{1}$ between the identity map $\mathbf{1}$ on $W(\Omega, T)$ and $\sigma\circ\Psi$. Hence $X(\Omega, T)$ is homotopy equivalent to the convex space $W(\Omega, T)$, hence it is contractible. 
\end{proof}

We can extend this result to spaces of geodesic triangulations of convex polygons in other geometries of constant curvature. 

\begin{corollary}
Assume $\Omega$ is a hyperbolic convex polygon, or a spherical convex polygon contained in an open hemisphere, and $T$ is a triangulation of $\Omega$. Then the space of geodesic triangulations $X(\Omega, T)$ is contractible. 
\end{corollary}
\begin{proof}
For a hyperbolic convex polygon $\Omega_H$, we embed it in the Klein disk model of the hyperbolic plane so that all the edges of $\Omega_H$ are line segments with respect to the Euclidean metric, inducing a convex polygon $\Omega$ in the Euclidean plane. There is a homeomorphism between the space $X(\Omega_H, T)$ and $X(\Omega, T)$, induced by the identity map of the disk. Hence the space of hyperbolic geodesic triangulations $X(\Omega_H, T)$ is contractible.

Similarly, if $\Omega_S$ is a spherical convex polygon contained in a hemisphere, we can apply the \textit{gnomonic transformation} from the center of the 2-sphere to the plane tangent to the center of the hemisphere containing $\Omega_S$. Then $\Omega_S$ is mapped to a convex polygon $\Omega$ in this plane under the gnomonic transformation. This transformation preserves the incidence and maps geodesic arcs in hemisphere to line segments in $\Omega$. Hence it induces a homeomorphism between $X(\Omega_S, T)$ and $X(\Omega, T)$.
\end{proof}

\section{Spaces of Geodesic Triangulations with non-trivial Topology}
In this section, we construct examples of spaces of geodesic triangulations with non-trivial $n$-th homotopy groups for each $n>0$. We first describe the building block of these constructions.

\subsection{The building block polygon $\mathcal{P}$}	
The building block is the polygon $\mathcal{P}$  in the Figure 4. The triangulation $T$ of $\mathcal{P}$ is given in Figure 4 with three interior vertices $P$, $L$, and $R$. For simplicity, in the remaining part of this paper, we only draw a part of the triangulation shown in Figure 4(B) instead of the full triangulation shown in Figure 4(A). Notice that we can add edges back to produce the full triangulation, once the positions of the interior vertices are fixed.

\begin{figure}[H]
  \begin{subfigure}{0.5\textwidth}
    \includegraphics[width=\linewidth]{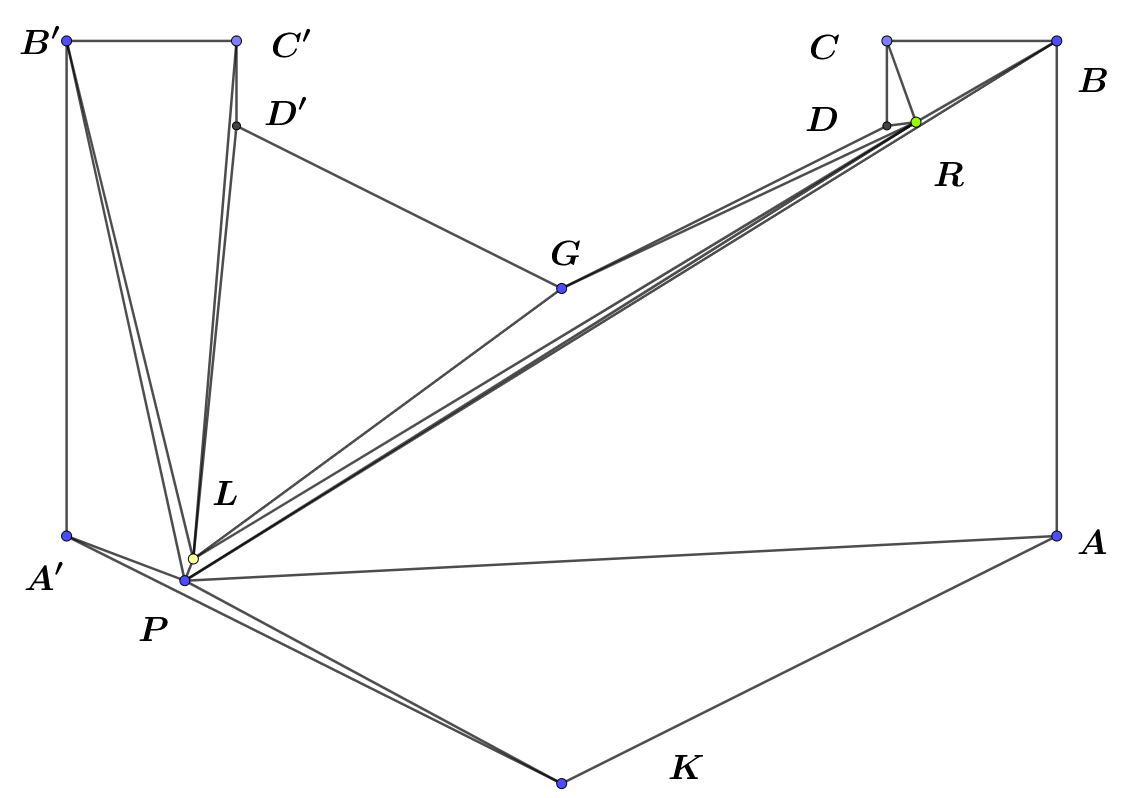}
    \caption{Triangulation of $\mathcal{P}$} \label{fig:1b}
  \end{subfigure}%
  \hspace*{\fill}   
  \begin{subfigure}{0.5\textwidth}
    \includegraphics[width=\linewidth]{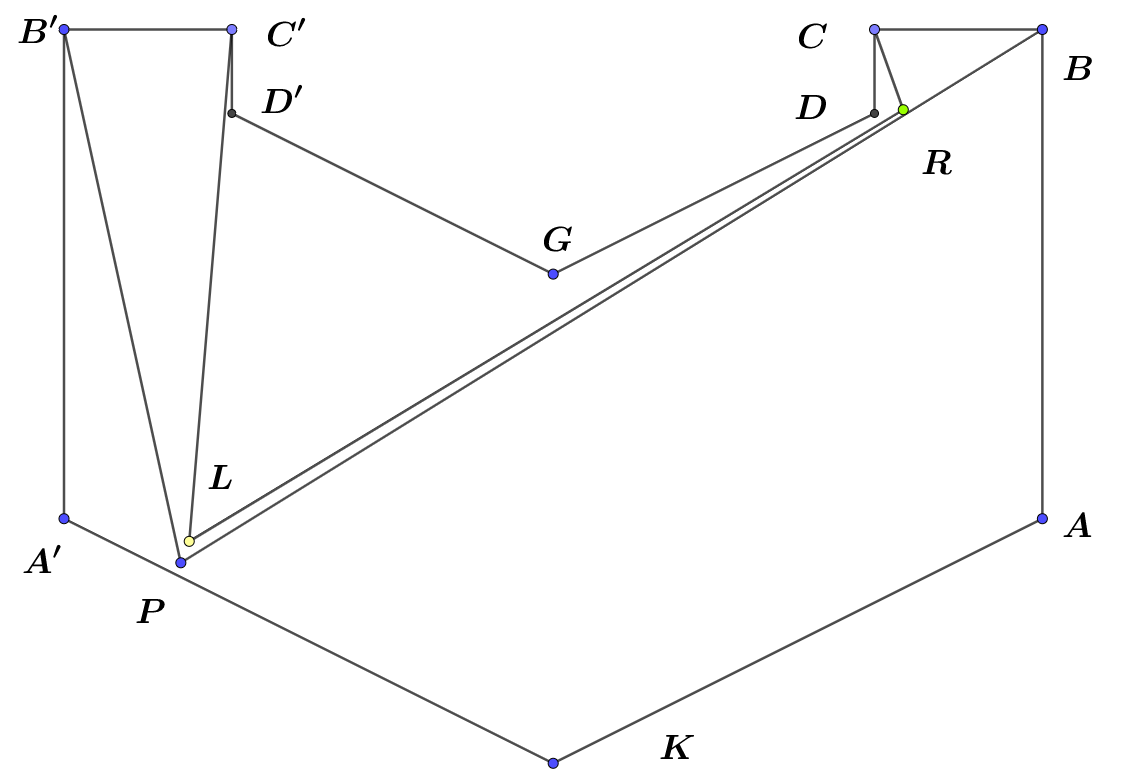}
    \caption{Part of the triangulation of $\mathcal{P}$.} \label{fig:1c}
  \end{subfigure}

\caption{The polygon $\mathcal{P}$ and its triangulation.} \label{fig:1}
\end{figure}

Set $c = (2+\sqrt{2})/(3+\sqrt{2})$. The coordinates of boundary vertices of $\mathcal{P}$ are 
$$A = \begin{bmatrix}1\\0 \end{bmatrix}, B = \begin{bmatrix}1\\1 \end{bmatrix}, C = \begin{bmatrix}c\\1 \end{bmatrix}, D = \begin{bmatrix} c\\ (c+1)/2 \end{bmatrix},  G = \begin{bmatrix} 0\\ 1/2 \end{bmatrix},  K = \begin{bmatrix}0\\-1/2 \end{bmatrix}.$$
Reflect the vertices $A$, $B$, $C$, and $D$ about $y$-axis to determine the remaining boundary vertices $A'$, $B'$, $C'$, and $D'$.

The idea of the construction of $\mathcal{P}$ originates from the example given by Bing and Starbird \cite{bing1978linear}. Let us consider an isosceles wedge $\angle A'KA$ in Figure 5 defined by

$$A =  \begin{bmatrix}1\\ 0\end{bmatrix}, \quad A' = \begin{bmatrix}-1\\ 0\end{bmatrix}, \quad K  = \begin{bmatrix}0\\ -k\end{bmatrix}, \quad k>0.$$

\begin{figure}[H]
\centering
\begin{subfigure}[b]{.49\linewidth}
\includegraphics[width=\linewidth]{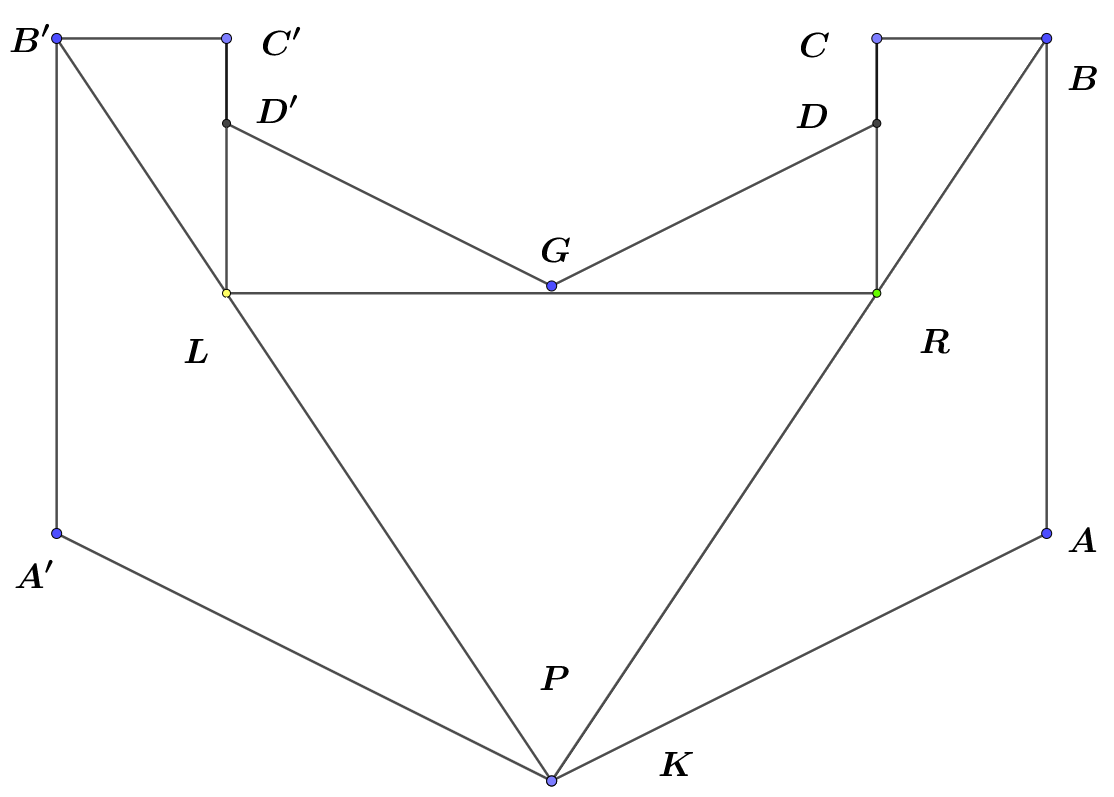}
\caption{$t = 0$}\label{fig:A}
\end{subfigure}
\begin{subfigure}[b]{.49\linewidth}
\includegraphics[width=\linewidth]{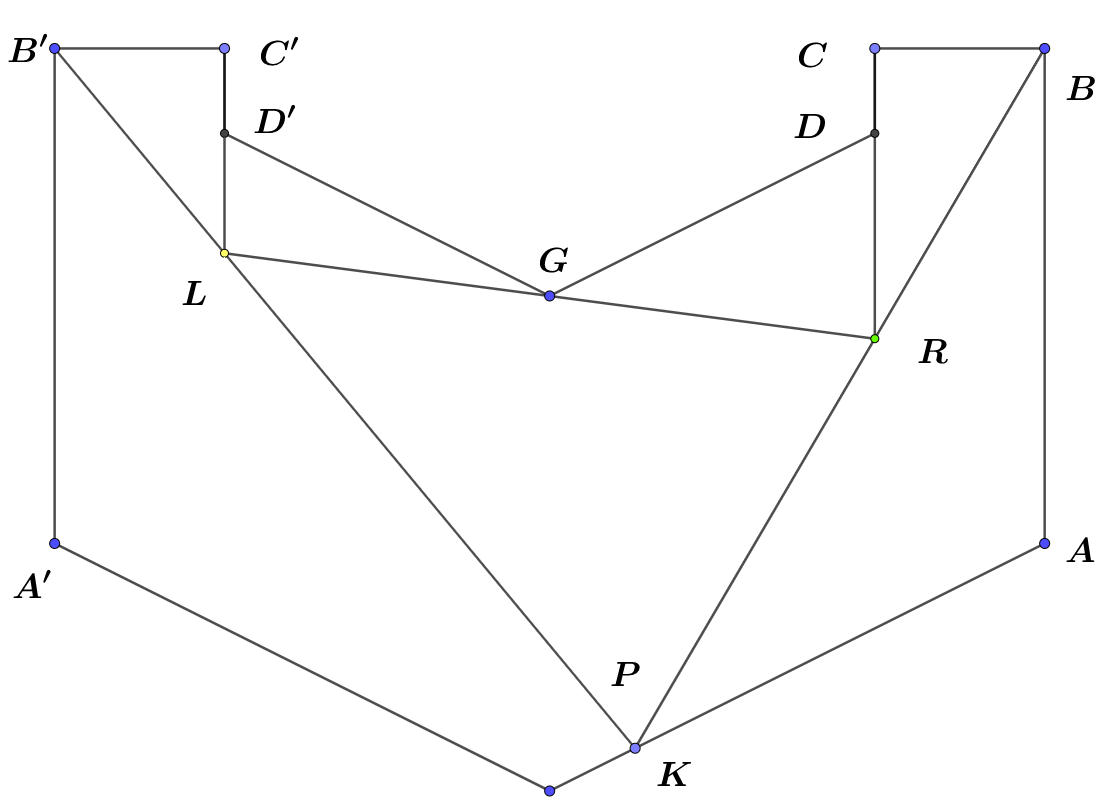}
\caption{$t=t_0$}\label{fig:B}
\end{subfigure}

\begin{subfigure}[b]{.49\linewidth}
\includegraphics[width=\linewidth]{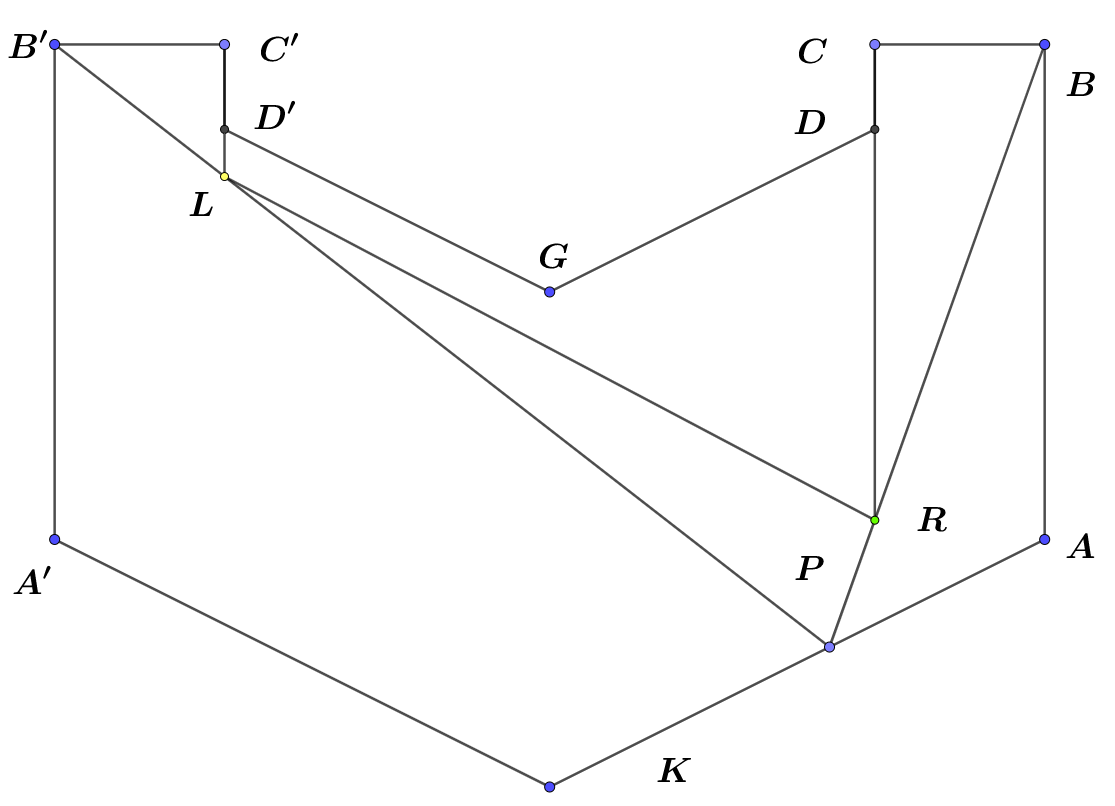}
\caption{$t_0<t<c$}\label{fig:C}
\end{subfigure}
\begin{subfigure}[b]{.49\linewidth}
\includegraphics[width=\linewidth]{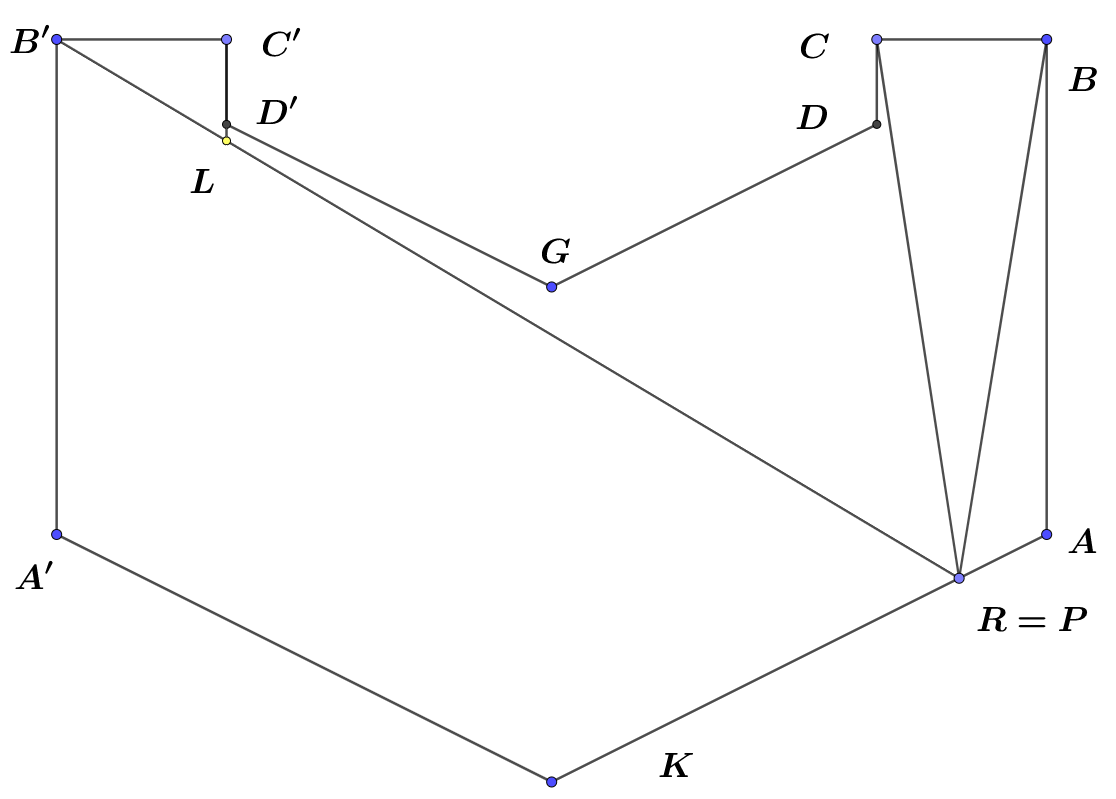}
\caption{$c<t<1$}\label{fig:D}
\end{subfigure}
\caption{The vertex $P$ moves along $KA$.}
\label{fig:mph}
\end{figure}

Let the vertex $P = (t, kt-k)$ move along the edge $KA$ for $t\in (0, 1)$. The lines $PB$ and $PB'$ are
$$y - 1 = \frac{1 - k(t-1)}{1-t}(x-1) \text{ for } PB, \quad y -1 = \frac{k(t-1)-1}{t+1}(x+1) \text{ for } PB'.$$
Set $R$ as the intersection of $PB$ with $x = c$, and $L$ as the intersection of $PB'$ with $x = -c$. The line $LR$ is given by 
$$\frac{y - (1 + (k(t-1)-1)(1-c)/(t+1))}{(1/(1-t) + k)(c-1) - (k(t-1)-1)(1-c)/(t+1)} = \frac{x+c}{2c}.$$
The $y$-intercept $H$ of $LR$ is 
$$h(t, k) = (c-1)\frac{t^2(k+1) - kt}{(1-t^2)} + 1 + (k+1)(c-1).$$

Note that the function $h(t, k)$ is not a monotonic function of $t$, but a monotonic decreasing function of $k$. The maximum of $h(t, k)$ in $0\leq t \leq c$ given below, denoted as $\mathbf{h}(k)$, is a continuous monotonically decreasing function of $k$
$$\mathbf{h}(k) = (c-1)(k + \frac{(\sqrt{2k+1} + 1)^2}{4}) + 1.$$
This corresponds to a clear geometric interpretation: if $k$ increases, the vertex $K$ moves downward, then the vertex $L$ and $R$ moves downward, so $\mathbf{h}(k)$ decreases. 

The polygon $\mathcal{P}$ corresponds to the case $k = 1/2$. The function $h(t, 1/2)$ achieves its maximum at $t_0 = 3-2\sqrt{2}$, with the maximal value 
$$\mathbf{h}(\frac{1}{2}) = (\frac{\sqrt{2}}{2} + \frac{3}{4})c + \frac{1}{4} - \frac{\sqrt{2}}{2} = \frac{1}{2}.$$

Recall that $G = (0, 1/2)^T$. This implies that as $P$ moves along $KA$, the edge $LR$ stays below $G$ shown in Figure 5(A) and 5(C), except at $t = t_0$, when $G$ lies on $LR$ shown in  Figure 5(B).

We highlight the properties of $\mathcal{P}$ by Lemma 4.2 below. It states that the space $X(\mathcal{P}, T)$ is not path-connected, but becomes path-connected after perturbations of vertices $G$ and $K$. We characterize these perturbations in the following definition. 

\begin{definition}
Given the triangulated polygon $(\mathcal{P}, T)$, let 
$$K' = K + \begin{bmatrix}0\\\epsilon\end{bmatrix}, \quad G' = G + \begin{bmatrix}0\\\delta\end{bmatrix}, $$
be a vertical perturbation of $K$ and $G$, and  $\mathcal{P}'$ the polygon after perturbation.  An \textit{admissible perturbation} of $K$ and $G$ is defined as one of the following:
\begin{itemize}
	\item $\epsilon<0$ and $\delta \geq 0$;
	\item $\delta>0$ and $\epsilon \leq 0$;
	\item $\epsilon >0$ and $\delta \geq 3\epsilon$;
	\item $\delta <0$ and $\epsilon \leq  3\delta$.
\end{itemize}
A \textit{forbidden perturbation} of $K$ and $G$ is defined as  one of the following:
\begin{itemize}
	\item $\epsilon=0$ and $\delta < 0$;
	\item $\delta=0 $ and $\epsilon > 0$;
\end{itemize}
\end{definition}

We will see that if $\mathcal{P}'$ is an admissible perturbation of $\mathcal{P}$, we can slide the vertex $P$ from the left to right. Let $\textbf{proj}_x^P$ be the continuous projection from $X(\mathcal{P}, T)$ to the $x$-coordinate of the vertex $P$. 

\begin{lemma}
The space $X(\mathcal{P}, T)$ is not empty. Moreover, $X(\mathcal{P}, T)$ is not path-connected. If $\mathcal{P}'$ is the polygon after an admissible perturbation of $\mathcal{P}$,
then there exists a path $\gamma(t)$ for $t\in[-c, c]$ in $X(\mathcal{P}', T)$ such that $\textbf{proj}_x^P(\gamma(t)) = t$ for $t\in[-c, c]$. On the other hand, if $\mathcal{P}'$ is the polygon after a forbidden perturbation of $\mathcal{P}$, then $X(\mathcal{P}', T)$ is not path-connected.
\end{lemma}

The geometric intuition of admissible perturbations is that we can move vertex $K$ downward or move vertex $G$ upward so that $X(\mathcal{P}', T)$ is connected. If we move $K$ upward, $G$ needs to be moved upward by a certain amount so that $X(\mathcal{P}', T)$ is connected. On the other hand, if we move $K$ upward or $Q$ downward, $X(\mathcal{P}', T)$ is not path-connected.  

\begin{proof}
We will show that 

$$\textbf{proj}_x^P(X(\mathcal{P}, T)) = (-1, -t_0)\cup (-t_0, t_0) \cup (t_0, 1).$$

Let $P = (t, 1/2t - 1/2)$ move along $KA$. If $t \neq t_0 = 3-2\sqrt{2}$, then $LR$ defined above is below $G$, hence we can displace the vertex $P$ vertically by a small distance into $\mathcal{P}$. Then move $R$ along a small vector pointing to the sector $\angle BRC$ and move $L$ along a small vector pointing to the sector $\angle B'LC'$ as shown in Figure 5. Then $G$ stays above $LR$ after the displacement of $L$ and $R$.

Since $L$ and $R$ move continuously as $P$ moves along $KA$, we can perturb $P$, $L$, and $R$ to construct a continuous map $\eta(t)$ for $t \in [0, t_0)\cup (t_0, c]$ in $X(\mathcal{P}, T)$ with $\textbf{proj}_x^P(\eta(t)) = t$. As mentioned before,  the segment $LR$ intersects with $G$ if $t = t_0$, so we can't move $P$ vertically to the interior of $\mathcal{P}$. This implies that there is no geodesic  triangulation $\tau$ in $X(\mathcal{P}, T)$ with $\textbf{proj}_x^P(\tau) = t_0$.

The vertices $B$, $D$, $G$, and $A'$ in $\mathcal{P}$ are chosen to be collinear. Hence the line segment connecting $B$ and any interior point of the segment $KA$ is contained in $\mathcal{P}$. For $c< t<1$, set $R = P$, then $LR = LP$ stays below $G$ and is contained in $\mathcal{P}$.  We can apply similar displacement of $P$, $L$, and $R$ to generate a geodesic triangulation $\tau$ with $\textbf{proj}_x^P(\tau) = t$. This shows that $(c, 1) \subset\textbf{proj}_x^P(X(\mathcal{P}, T))$. By symmetry, it follows that $$\textbf{proj}_x^P(X(\mathcal{P}, T)) = (-1, -t_0)\cup (-t_0, t_0) \cup (t_0, 1).$$

On the other hand, for $k\in(0, \infty)$, the derivative of $\mathbf{h}(k)$ is bounded by 
$$\frac{1}{3}\leq |\mathbf{h}'(k)| = |(c-1)(1+ \frac{1}{2}(1 + \frac{1}{\sqrt{2k+1}}))| \leq \frac{1}{2}.$$

Then $|\mathbf{h}'| \leq 3$ and $|(\mathbf{h}^{-1})'| \leq 3$. It implies that if $\epsilon>0$ $$\mathbf{h}(\frac{1}{2} - \epsilon) \leq \frac{1}{2} + 3\epsilon.$$

This means that the segment $LR$ is always below $G' = (0, 1/2 + \delta)$, if we apply one of the four cases of admissible perturbations. By similar displacements of vertices $P$, $L$, and $R$ as before, we can construct a continuous path $\gamma(t)$ for $t\in [0, 1)$ in $X(\mathcal{P}', T)$ with $\textbf{proj}_x^P(\gamma(t)) = t$. 

Extending this path by symmetry, we can assume this path is defined on $(-1, 1)$, and $\gamma(-t)$ is the reflection of $\gamma(t)$ above $y$-axis. Moreover, we can assume that the vertices $P$, $L$, and $R$ given by $\gamma(c)$ also produce a geodesic triangulation on $\mathcal{P}$. Then the restriction of $\gamma(t)$ on $[-c, c]$ is the desired path. 

Finally, if $\mathcal{P}'$ is a forbidden perturbation of $\mathcal{P}$, the maximal $y$-intercept $H$ of $LR$ as $P$ moves along $KA$ exceeds the height of $G'$. In the first case of forbidden perturbations, $K$ is fixed and $G$ is moved downward, so $LR$ intersects with $G'$ when $t = t_0$. In the second case of forbidden perturbations, $G$ is fixed and $K$ is moved upward. Recall that $h(t_0, k)$ is a monotonic decreasing function of $k$. Then the $y$-intercept $H$ of $LR$ when $t = t_0$ lies above $G$. Hence in both cases, $LR$ intersects with $G'$ when $t = t_0$. By a similar argument as before, $t= t_0$ is not in $\textbf{proj}_x^P(X(\mathcal{P}', T))$, which implies that $X(\mathcal{P}', T)$ is not path-connected. 
\end{proof}

The key to constructing the path $\gamma$ is that the segment $LR$ never intersects $G$ as $P$ moves. Following the argument in Bing and Starbird \cite{bing1978linear}, one can show that $X(\mathcal{P}', T)$ is path-connected for an admissible perturbation $\mathcal{P}'$ of $\mathcal{P}$.

\subsection{The main idea of the construction}
The main idea to construct a polygon $\Omega^n$ with non-trivial $n$-th homotopy group is to stack $n+1$ copies of $\mathcal{P}$ together.  The polygon $\Omega^1$ is shown in Figure 6. Again, we only draw part of the triangulation $T^1$ as before. The full triangulation can be constructed by adding edges back once the interior vertices are fixed. 

\begin{figure}[H]
  \includegraphics[width=0.8\linewidth]{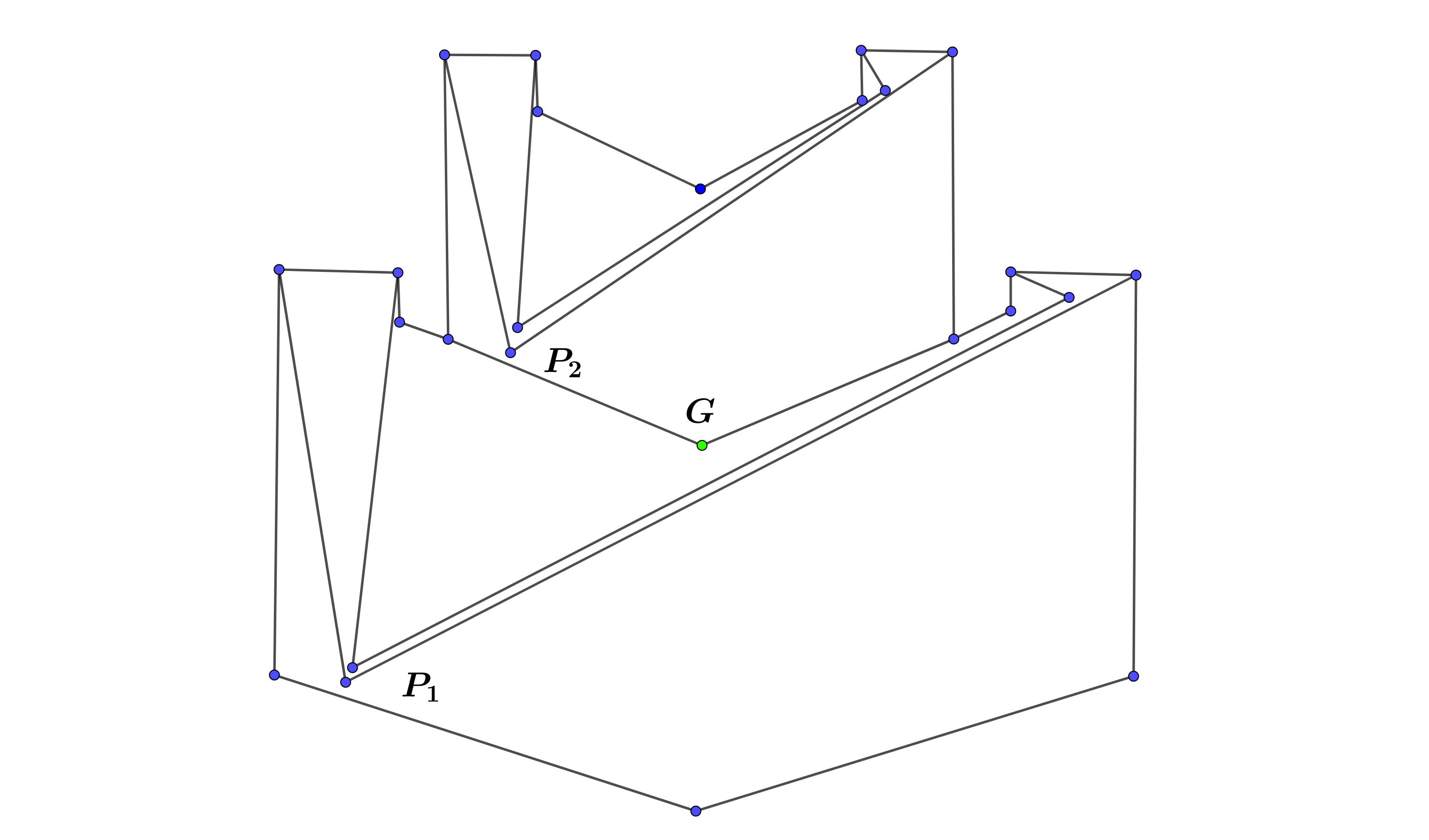}
      \caption{The polygon $\Omega^1$ with non-trivial $\pi_1(X(\Omega^1, T^1))$.}
\end{figure}

We illustrate the idea informally when $n = 1$. To show that $X(\Omega^1, T^1)$ has a non-trivial fundamental group, we construct a non-trivial loop in $X(\Omega^1, T^1)$ using the two non-homotopic paths in Figure 7. 

\begin{figure}
  \includegraphics[width=1\linewidth]{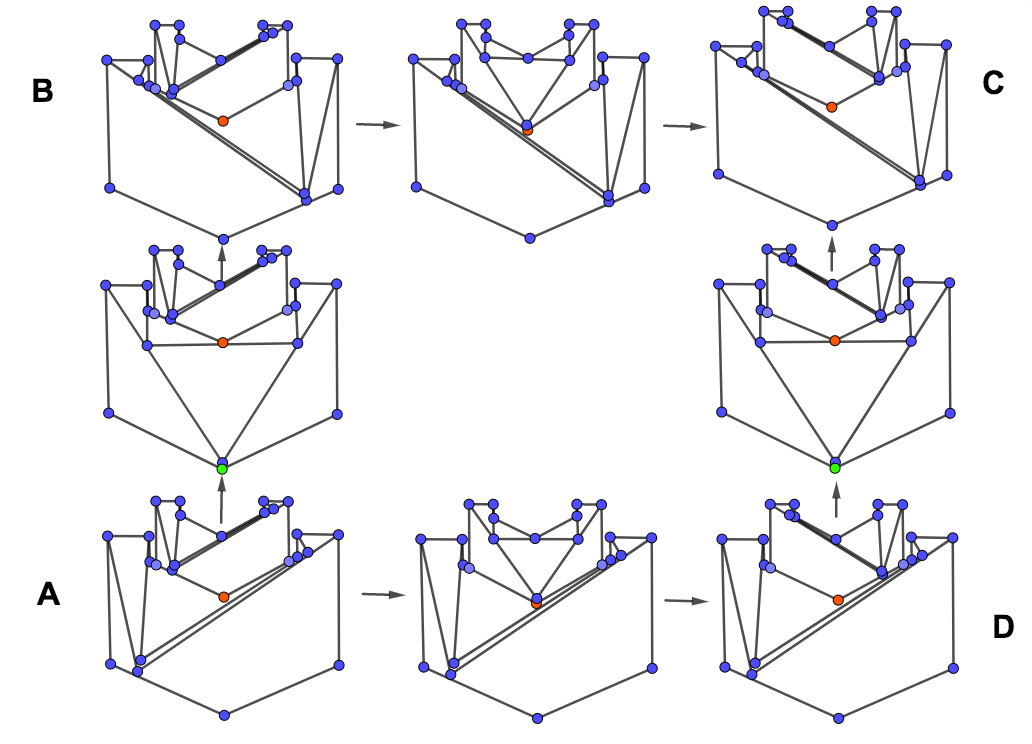}
      \caption{Two non-homotopic  paths $\mathbf{A}\to\mathbf{D}\to\mathbf{C}$ and $\mathbf{A}\to\mathbf{B}\to\mathbf{C}$.}
\end{figure}

Starting with the configuration $\mathbf{A}$, we construct two paths connecting $\mathbf{A}$ and $\mathbf{C}$, shown as the paths $\mathbf{A} \to \mathbf{B}\to\mathbf{C}$ and $\mathbf{A} \to \mathbf{D}\to\mathbf{C}$. In the first path, move the vertex $G$ upward so that $\mathcal{P}_1'$ is an admissible perturbation of $\mathcal{P}_1$. Then we can slide the vertex $P_1$ from the left to right to reach the configuration $\mathbf{B}$. Then move $G$ down so that $\mathcal{P}_2'$ is an admissible perturbation of $\mathcal{P}_2$. Slide the vertex $P_2$ from the left to right to the configuration $\mathbf{C}$. In the second path, the order is reversed: first move $G$ down to slide $P_2$ from the left to right to the configuration $\mathbf{D}$, then move $G$ up to slide $P_1$ to the configuration $\mathbf{C}$.

We claim that the two paths above are not homotopic. Equivalently, the loop $\mathbf{A} \to \mathbf{B}\to\mathbf{C} \to\mathbf{D}\to\mathbf{A}$ represents a non-trivial element in $\pi_1(X(\Omega^1, T^1))$. 

It follows from two observations: if we project this loop to the $x$-coordinate of the vertex $P_1$ and $x$-coordinate of the vertex $P_2$ with suitable rescaling, the image of the loop is the cycle in the plane in Figure 8. 

The point $(t_0, t_0)$ lies inside the cycle, but is not in the image of the projection of $X(\Omega^1, T^1)$. If the $x$-coordinate of $P_1$ is $t_0$, then the vertex $G$ has to be moved upward by Lemma 4.2. Similarly, the fact that the $x$-coordinate of $P_2$ equals $t_0$ implies that $G$ has to be moved downward. This is a contradiction. 

Hence this cycle is not trivial in $\textbf{proj}_x^P(X(\Omega^1, T^1))$. Hence the loop $\mathbf{A} \to \mathbf{B}\to\mathbf{C} \to\mathbf{D}\to\mathbf{A}$ is not trivial in $X(\Omega^1, T^1)$. A rigorous proof is given in the next section. 

\begin{figure}[H]
  \includegraphics[width=0.5\linewidth]{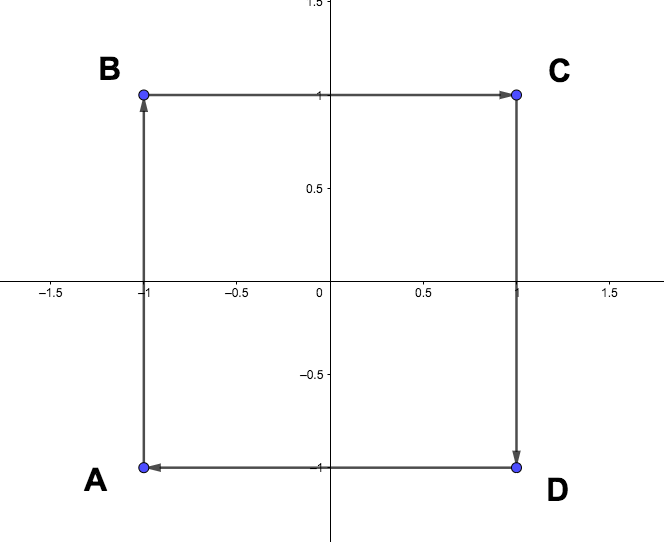}
      \caption{The projection of the loop in $\mathbb{R}^2$.}
\end{figure}

For any $n\geq 1$, we stack $n+1$ copies of $\mathcal{P}$ together to construct the polygons $\Omega^n$ with triangulation $T^n$. Denote these copies of $\mathcal{P}$ by $\mathcal{P}_i$ for $i \in I=\{ 1, 2, \cdots, n+1\}$. Let $P_i$ be the corresponding vertex $P$ of $\mathcal{P}$ in each $\mathcal{P}_i$ for $i \in I$.

Start with the first copy $\mathcal{P}_1 = \mathcal{P}$. Given the fact $\angle A'KA+\angle D'GD = 360^{\circ}$, we can inductively identify the vertex $K_{i+1}$ of $\mathcal{P}_{i+1}$ with the vertex $G_i$ of $\mathcal{P}_i$, and the wedge $\angle A'_{i+1}K_{i+1}A_{i+1}$ with part of $\angle D'_{i}G_iD_i$ by similarity transformations of the plane
$$g_i(x, y) = a_i(x, y) + (0, b_i)\quad 0<a_i <1, \quad 0<b_i, \quad  i =  2, 3, \cdots, n+1 $$ 
 to construct $\Omega^n$. Define $\mathcal{P}_i$ in $\Omega^n$ to be the image of $\mathcal{P}$ under the similarity transformations $g_i$. Note that $g_1$ is the identity map of $\mathbb{R}^2$ with $a_1 = 1$ and $b_1 = 0$. Each $\mathcal{P}_i$ is a subgraph of the $1$-skeleton of $T^n.$
 
Figure 9 shows the polygon $\Omega^n$ with three copies of $\mathcal{P}$ stacked together. Notice that there is a small gap between two consecutive copies of $\mathcal{P}$ by proper choices of similarity transformations. We will show that for any $n>0$, $X(\Omega^n, T^n)$ has non-trivial $n$-th homotopy group. 

\begin{figure}[H]
  \includegraphics[width=0.6\linewidth]{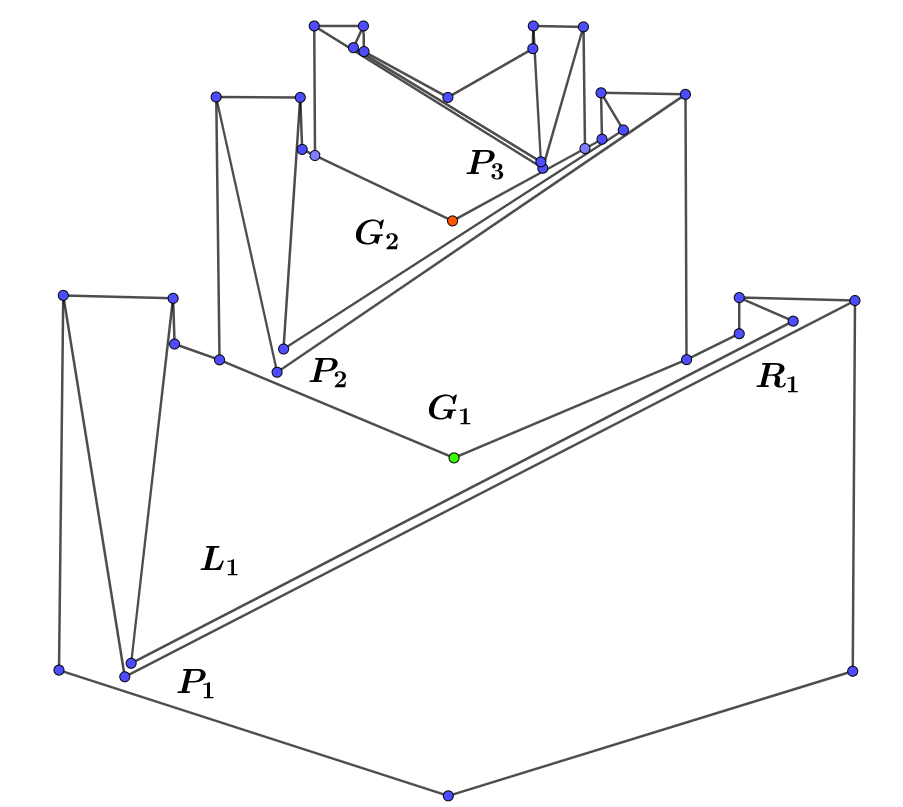}
      \caption{A configuration in $X(\Omega^2, T^2)$ with $\pi_2(X(\Omega^2, T^2))$ non-trivial.}
\end{figure}

\subsection{Proof of Theorem 1.1}
To prove the new result of this paper, 
we first establish the existence of admissible perturbations of $\mathcal{P}_i$ for $i\in J$, where $J$ is strict subset of $I$.
\begin{lemma}
Given $\epsilon>0$, for any strict subset $J$ of $I$, there exist admissible perturbations of vertices $K_i$ and $G_i$ for $i\in I$, with the two boundary vertices $K_1$ and $G_{n+1}$ fixed, such that if $j\in J$, the space $X(\mathcal{P}_j', T)$ is path-connected. Moreover, the displacements of all vertices $K_i$ and $G_i$ are smaller than $\epsilon$.
\end{lemma}
\begin{proof}
 
Let $\{i, i+1, \cdots, i+m\}$ be a maximal chain of consecutive integers in $J$ with $m\geq 0$. If $i \neq 1$, move $K_{j}$ downward by $\epsilon/3^{j-i}$ for $j \in\{ i, i+1, \cdots, i+m\}$, producing admissible perturbations of $\mathcal{P}_j$ for $j \in\{ i, i+1, \cdots, i+m\}$. If the maximal chain of consecutive integers is $\{1, 2, \cdots, m\}$, then move $G_j$ upward by $\epsilon/3^{m-j}$ for $j \in \{1, 2, \cdots, m\}$. It also produces admissible perturbations of $\mathcal{P}_j$ for $j \in \{1, 2, \cdots, m\}$.

In general, there are many maximal chains of consecutive integers with different lengths in $J$. We can perform the admissible perturbations above separately for each maximal chain. Notice that since $J$ is a strict subset of $I$, the boundary vertices $K_1$ and $G_{n+1}$ can be fixed.  

Therefore, each of the spaces $X(\mathcal{P}_j', T)$ is path-connected for $j\in J$ after the admissible perturbations on $G_i$ and $K_i$ above, and the displacements of $K_i$ and $G_i$ are smaller than $\epsilon$.
\end{proof}

We construct geodesic triangulations on $\Omega^n$ by combining  geodesic triangulations of each $\mathcal{P}_i$, or their admissible perturbations $\mathcal{P}_i'$ for $i\in I$. Define a path in $X(\mathcal{P}'_i, T)$ by 
$$\gamma_i(t) = g_i(\gamma(t)) = a_i\gamma(t) + b_i,  \quad  t\in[-c, c], i\in I.$$
The path $\gamma(t)$ is constructed in Lemma 4.2 with $c = (2+\sqrt{2})/(3+\sqrt{2})$. A geodesic triangulation in $X(\mathcal{P}'_i, T)$ can be regarded as a part of a geodesic triangulation in $X(\Omega^n, T^n)$, determining the positions of $P_i$, $L_i$, and $R_i$. Start with a geodesic triangulation in $X(\Omega^n, T^n)$ such that the positions of $P_i$, $L_i$, and $R_i$ are given by $\gamma_i(-c)$. The path $\gamma_i$ in $X(\mathcal{P}'_i, T)$ produces a path in $X(\Omega^n, T^n)$ with $\textbf{proj}_x^{P_i}(\gamma_i(t)) = a_it$ for $t\in [-c, c]$. Notice that this path only moves three vertices $P_i$, $L_i$, and $R_i$ in $\mathcal{P}_i'$ and fixes all the other vertices.

Combining paths above, we can simultaneously move vertices $P_i$, $R_i$, and $L_i$ after suitable admissible perturbations. Start with a geodesic triangulation in $X(\Omega^n, T^n)$, whose vertices $P_i$, $R_i$, and $L_i$ are determined by $\gamma_i(-c)$. By Lemma 4.3, for any strict subset $J = \{i_1, i_2, \cdots, i_m\}$ of  $I$, there exist admissible perturbations of $\mathcal{P}_{i_k}$ such that $X(\mathcal{P}_{i_k}', T)$ is path-connected for $i_k \in J$. Then each $\gamma_{i_k}(t)$ for $i_k\in J$ defines a path in $X(\Omega^n, T^n)$, which restricts to the identity map except in $\mathcal{P}_{i_k}'$. 

For any $(y_{1}, \cdots, y_{m})$ with $|y_k| \leq c$, we define $\Gamma(\pm c, \cdots, \pm c, y_{1}, y_{2}, \cdots, y_{m})$ to be the geodesic triangulation in $X(\Omega^n, T^n)$ which is determined by $\gamma_{i_k}(y_k)$ for $i_k\in J$ and $\gamma_i(\pm c)$ for $i\not\in J$. This means that the positions of $P_{i_k}$, $L_{i_k}$, and $R_{i_k}$ are determined by paths $\gamma_{i_k}(y_k)$ if $i_k \in J$, and all the other vertices are fixed at $\gamma_{i}(\pm c)$ for $i\not\in J$.  Using the map $\Gamma$, we prove that 

\begin{proposition}
The homotopy group $\pi_n(X(\Omega^n, T^n))$ is not trivial.
\end{proposition}

\begin{proof}

We construct a continuous map $f$ from the boundary $\partial I^{n+1}$ of an $(n+1)$-dimensional cube $I^{n+1}$ in $\mathbb{R}^{n+1}$ 
$$I^{n+1} = \{(x_1, \cdots, x_{n+1}) \quad | \quad |x_i|\leq 1, \quad  i \in I\}$$
to $X(\Omega^n, T^n)$, representing a non-trivial element in $\pi_n(X(\Omega^n, T^n))$. 

The set $\partial I^{n+1}$ consists of $k$-dimensional cubes for $0\leq k\leq n$. The $n$-dimensional cube corresponds to the facets  
$$F_i^\pm = \{(x_1, \cdots,x_{i-1}, \pm 1, x_{i+1},\cdots, x_{n+1}) \quad| \quad |x_i|\leq 1, \quad  i \in I\}.$$
Notice that all the $k$-dimensional cubes can be represented as the intersections of a collection of facets. We will define $f$ by induction on the dimension of cubes.

The $0$-dimensional cubes correspond to the vertices $(x_1, x_2, \cdots, x_{n+1})$ of $\partial I^{n+1}$ with $x_i = \pm 1$ for $i\in I$. We define the map $f$ on $0$-dimensional cubes by
$$f(x_1, x_2, \cdots, x_{n+1}) =\Gamma(cx_1, cx_2, \cdots, cx_{n+1}).$$
The positions of vertices of $\Gamma(cx_1, cx_2, \cdots, cx_{n+1})$ are determined by $\gamma_{i}(\pm c)$ for each $i\in I$. Notice that if $x_i = 1$, the vertex $P_i$ lies on the right side of $\mathcal{P}_i$ with its $x$-coordinate equal to $a_ic$.

Inductively, assume that $f$ is defined on all the $k$-dimensional cubes in $\partial I^{n+1}$ with $k<n$. We extend $f$ to all $(k+1)$-dimensional cubes in the following two steps.

Up to permutations of indices, let us extend $f$ to $\mathcal{H}$, which is the intersection of the facets $F_i^+$ for $i = 1, \cdots, n-k$, representing a $(k+1)$-dimensional cube

$$\mathcal{H} = \bigcap_{i = 1}^{n-k} F^+_i = \{(1, \cdots, 1, x_{n-k+1},\cdots, x_{n+1}) \quad| \quad |x_i|\leq 1, \quad i \in I.\}$$

Consider the rescaling of the interior of $\mathcal{H}$ by $c$ from the center of $\mathcal{H}$

$$c\mathcal{H} = \bigcap_{i = 1}^{n-k} F^+_i = \{(1, \cdots, 1, x_{n-k+1},\cdots, x_{n+1}) \quad| \quad |x_i|< c, \quad  i \in I.\}$$

In the first step, we extend $f$ from $\partial \mathcal{H}$ to $\mathcal{H} - c\mathcal{H}$ using radial segments from the center, as shown in Figure 10. In each segment $\mathcal{L}$, we fix vertices $P_i$, $L_i$, and $R_i$ in each  $\mathcal{P}_i$, and move vertically the vertices $K_i$ and $G_i$ from the unique configuration $f(\partial \mathcal{H} \cap \mathcal{L})$ to  admissible perturbations of $\mathcal{P}_i$ for $i = n-k+1, \cdots, n+1$.  By Lemma 4.3, such admissible permutations exist, and we can choose $\epsilon$ small such that no intersections occurs during the vertical displacements of vertices $K_i$ and $G_i$. 

\begin{figure}[H]
  \includegraphics[width=0.5\linewidth]{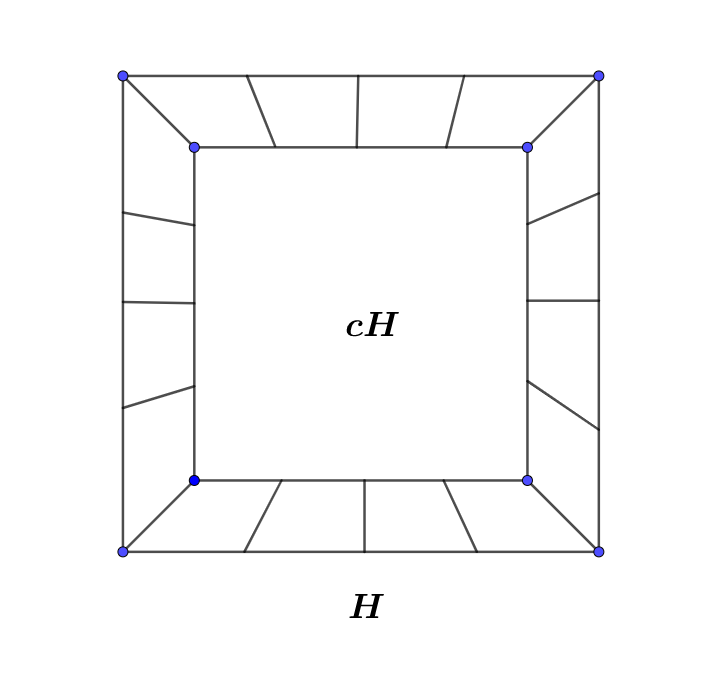}
  \caption{Extension of the embedding $f$.}
\end{figure}

In the second step, we extend $f$ to $c\mathcal{H}$. Notice that $f$ has been extended to the boundary of $c\mathcal{H}$ in the first step.  At boundary points of $c\mathcal{H}$, the spaces $X(\mathcal{P}_i', T)$ for $i = n-k+1, \cdots, n+1$ are path-connected. Define $f$ on $c\mathcal{H}$ by
$$f((1, \cdots, 1, x_{n-k+1}, \cdots, x_{n+1})) = \Gamma(c, \cdots, c, x_{n-k+1}, \cdots, x_{n+1}),    \quad |x_i|\leq c. $$
This determines a continuous extension of $f$ on $\mathcal{H}$. Similar formulas hold for the other $(k+1)$-dimensional faces. By induction, $f$ is extended to $\partial I^{n+1}$. 

We will show that $f$ can't be extended to $I^{n+1}$, because we can't move simultaneously vertices $P_i$ from left to right for all $i\in I$. This implies that $f$ represents a non-trivial element in $\pi_n(X(\Omega^n, T^n))$.

Define a continuous projection from $X(\Omega^n, T^n)$ with scaling $\Phi: X(\Omega^n, T^n) \to \mathbb{R}^{n+1}$ by
$$ \Phi(\tau) = (\frac{\textbf{proj}_x^{P_1}(\tau)}{a_1}, \frac{\textbf{proj}_x^{P_2}(\tau)}{a_2},\cdots, \frac{\textbf{proj}_x^{P_{n+1}}(\tau)}{a_{n+1}}),\quad \tau\in X(\Omega^n, T^n),$$
where $a_i$ is the coefficients of similarity transformations $g_i$ in the definition of $\Omega^n$, and $\textbf{proj}_x^{P_i}$ is the projection to the $x$-coordinate of the vertex $P_i$.

We claim that $\Phi(f(\partial I^{n+1})) =  c\partial I^{n+1} \subset \mathbb{R}^{n+1}$, where $c\partial I^{n+1}$ is a rescaling of $\partial I^{n+1}$ by $c$. Without loss of generality, consider the facet $F_1^+$ and its rescaling

$$cF_1^+ = \{(1, x_2, \cdots, x_{n+1}) \quad| \quad |x_i|\leq c, \quad  i \in I\}.$$
By the definition of $f$ on $cF_1^+$,  
$$f((1, x_2, \cdots, x_{n+1})) = \Gamma(c, x_2, \cdots,x_{n+1}), \quad |x_i|\leq c ,\quad i \in I\}.$$
Since $\textbf{proj}_x^{P_i}(\gamma_i(t)) = a_it$ for $t\in [-c, c]$, 
$$\Phi(f((1, x_2, \cdots, x_{n+1}))) = (c, x_2, \cdots, x_{n+1}) \quad |x_i|\leq c, \quad  i \in I\}.$$
Notice that on $F_1^+ - cF_1^+$, vertices $P_i$ are fixed for $i\in I$. This implies that $\Phi(f(\partial I^{n+1})) = c \partial I^{n+1}$. The restriction of $\Phi\circ f$ on $F_1^+$ is homotopic to the rescaling map of $F_1^+$ by the constant $c$. Hence $\Phi\circ f$ is a degree-one map from $\partial I^{n+1}$ to $c\partial I^{n+1}$. 

Based on this fact, we will show that $f(\partial I^{n+1})$ represents a non-trivial element in $X(\Omega^n, T^n)$. Equivalently, we show that $f$ can't be extended to $I^{n+1}$ in $X(\Omega^n, T^n)$.

Recall $t_0 = 3-2\sqrt{2}$. Set   
$$\vec{v} = (t_0, t_0, \cdots, t_0)\in\mathbb{R}^{n+1}.$$
Since $c = (2 + \sqrt{2})/(3 + \sqrt{2})>t_0$, the vector $\vec{v}$ is contained in $ cI^{n+1} = \Phi(f(\partial I^{n+1}))$.
To show that $f$ can't be extended to $I^{n+1}$, we show that $\vec{v}\not\in \Phi(X(\Omega^n, T^n))$. 

Suppose that there exists a geodesic triangulation $\eta \in X(\Omega^n, T^n)$ such that $\Phi(\eta) = \vec{v}$. Notice that $K_1$ and $G_{n+1}$ are fixed as boundary vertices. By Lemma 4.2, if $t_0\in \textbf{proj}_x^{P_1}(X(\Omega^n, T^n))$, we need to perform an admissible perturbation of the vertex $G_1$ to avoid intersections of edges in $\mathcal{P}_1$. Hence $G_1$ is moved upward, and inductively all the $G_i$ for $i \in I$ need to be moved upward. However, $G_{n+1}$ is fixed, so no such triangulation exists. 
\end{proof}

Theorem 1.1 follows from Proposition 4.3, since $(\Omega^n, T^n)$ provides the required polygons for $n>0$. This theorem resolves a problem proposed in \cite{connelly1983problems}, showing that the space of geodesic triangulations of a planar polygon could have complicated topology.

\section{Further Work}
The homotopy type of the space $X(S, T)$ of geodesic triangulations of a general closed surface $S$ with constant curvature remains unknown. The space of geodesic triangulations of the 2-sphere $X(S^2, T)$ was studied by Awartani-Henderson \cite{awartani1987spaces}, but its homotopy type remains open. If $S$ is a  hyperbolic surface, Hass and Scott \cite{hass2012simplicial} showed that $X(S, T)$ is contractible if $T$ is a one-vertex triangulation. It is conjectured in \cite{connelly1983problems} that $X(S, T)$ is homotopy equivalent to the group of isometries of $S$ homotopic to the identity, when $S$ is equipped with a metric with constant curvature. 

Another open question is whether $X(\Omega, T)$ can realize the homotopy type of any given finite CW complex. This universality property holds for configuration spaces of linkages \cite{kapovich2002universality}. 

\bibliography{ref} 
\bibliographystyle{amsplain}

\end{document}